\documentclass{amsart}
\usepackage{graphicx, color}
\usepackage{amscd}
\usepackage{amsmath,empheq}
\usepackage{amsfonts}
\usepackage{amssymb}
\usepackage{mathrsfs}
\usepackage[all]{xy}
\newtheorem{theorem}{Theorem}
\newtheorem{ex}{Example}

\newtheorem{prop}[theorem]{Proposition}
\newtheorem{remark}{Remark}
\newtheorem{corollary}[theorem]{Corollary}
\newtheorem{claim}{Claim}

\newenvironment{proof-sketch}{\noindent{\bf Sketch of Proof}\hspace*{1em}}{\qed\bigskip}
 
\everymath{\displaystyle} 
\newcommand{\RR}{\mathbb R}
\newcommand{\NN}{\mathbb N}

\renewcommand{\leq}{\leqslant}

\renewcommand{\geq}{\geqslant}
\begin{document}
\title[Positive solutions for nonlinear parametric singular Dirichlet problems]{Positive solutions for nonlinear parametric singular Dirichlet problems}
\author[N.S. Papageorgiou]{Nikolaos S. Papageorgiou}
\address[N.S. Papageorgiou]{National Technical University, Department of Mathematics,
				Zografou Campus, Athens 15780, Greece \& Institute of Mathematics, Physics and Mechanics, Jadranska 19, 1000 Ljubljana, Slovenia}
\email{\tt npapg@math.ntua.gr}
\author[V.D. R\u{a}dulescu]{Vicen\c{t}iu D. R\u{a}dulescu}
\address[V.D. R\u{a}dulescu]{Institute of Mathematics, Physics and Mechanics, Jadranska 19, 1000 Ljubljana, Slovenia \& Faculty of Applied Mathematics, AGH University of Science and Technology, al. Mickiewicza 30, 30-059 Krak\'ow, Poland \& Institute of Mathematics ``Simion Stoilow" of the Romanian Academy, P.O. Box 1-764,
          014700 Bucharest, Romania}
\email{\tt vicentiu.radulescu@imar.ro}
\author[D.D. Repov\v{s}]{Du\v{s}an D. Repov\v{s}}
\address[D.D. Repov\v{s}]{Faculty of Education and Faculty of Mathematics and Physics, University of Ljubljana, 1000 Ljubljana, Slovenia \& Institute of Mathematics, Physics and Mechanics, Jadranska 19, 1000 Ljubljana, Slovenia}
\email{\tt dusan.repovs@guest.arnes.si}
\begin{abstract}
We consider a nonlinear parametric Dirichlet problem driven by the $p$-Laplace differential operator and a reaction which has the competing effects of a parametric singular term and of a Carath\'eodory perturbation which is ($p-1$)-linear near $+\infty$. The problem is uniformly nonresonant with respect to the principal eigenvalue of $(-\Delta_p,W^{1,p}_0(\Omega))$. We look for positive solutions and prove a bifurcation-type theorem describing in an exact way the dependence of the set of positive solutions on the parameter $\lambda>0$.
\end{abstract}
\maketitle

\section{Introduction}

Let $\Omega \subseteq\RR^N$ be a bounded domain with  $C^2$-boundary $\partial\Omega$. In this paper we study the following nonlinear parametric singular Dirichlet problem:
\begin{equation}\tag{$P_{\lambda}$}\label{eqp}
	\left\{\begin{array}{l}
		-\Delta_pu(z)=\lambda u(z)^{-\gamma}+f(z,u(z))\ \mbox{in}\ \Omega,\\
		u|_{\partial\Omega}=0,\ u>0,\ \lambda>0,\ 0<\gamma<1.
	\end{array}\right\}
\end{equation}

In this problem, $\Delta_p$ denotes the $p$-Laplacian differential operator defined by
$$\Delta_pu={\rm div}\,(|Du|^{p-2}Du)\ \mbox{for all}\ u\in W^{1,p}_0(\Omega),\ 1<p<\infty.$$

On the right-hand side of \eqref{eqp} (the reaction of the problem), we have a parametric singular term $u\mapsto \lambda u^{-\gamma}$ with $\lambda>0$ being the parameter and $0<\gamma<1$. Also, there is a Carath\'eodory perturbation $f(z,x)$ (that is, for all $x\in\RR$ the mapping $z\mapsto f(z,x)$ is measurable and for almost all $z\in\Omega$ the mapping $x\mapsto f(z,x)$ is continuous). We assume that $f(z,\cdot)$ exhibits $(p-1)$-linear growth near $+\infty$.

 We are looking for positive solutions of problem \eqref{eqp}. Our aim is to describe in a precise way the dependence on the parameter $\lambda>0$ of the set of positive solutions.

 \textcolor{black}{
We prove a bifurcation-type property, which is the main result of our paper. Concerning the hypotheses $H(f)$ on the perturbation $f(z,x)$ and the other notation used in the statement of the theorem, we refer to Section~2.
The main result of the present paper is stated in the following theorem.}

\textcolor{black}{
{\bf Theorem A.} {\sl If hypotheses $H(f)$ hold, then there exists $\lambda^*\in(0,+\infty)$ such that
	\begin{itemize}
		\item [(a)] for every $\lambda\in(0,\lambda^*)$, problem \eqref{eqp} has at least two positive solutions
			$$
			u_\lambda,\hat{u}_\lambda\in {\rm int}\,C_+,\ u_\lambda\neq\hat{u}_\lambda, \ u_\lambda\leq\hat{u}_\lambda;
			$$
		\item [(b)] for $\lambda=\lambda^*$, problem \eqref{eqp} has at least one positive solution
			$$
			u^*_\lambda\in {\rm int}\,C_+;
			$$
		\item [(c)] for $\lambda>\lambda^*$, problem \eqref{eqp} has no positive solutions.
	\end{itemize}}
}

In the past, singular problems were studied in the context of semilinear equations (that is, $p=2$). We mention the works of Coclite \& Palmieri \cite{2}, Ghergu \& R\u{a}dulescu \cite{5}, Hirano, Saccon \& Shioji \cite{10}, Lair \& Shaker \cite{11}, Sun, Wu \& Long \cite{20}. A detailed bibliography and additional topics on the subject, can be found in the book of Ghergu \& R\u{a}dulescu \cite{6}. For nonlinear equations driven by the $p$-Laplacian, we mention the works of Giacomoni, Schindler \& Taka\v{c} \cite{7}, Papageorgiou, R\u{a}dulescu \& Repov\v{s} \cite{16, 16bis}, Papageorgiou \& Smyrlis \cite{17}, Perera \& Zhang \cite{18}. Of the aforementioned papers, closest to our work here is that of Papageorgiou \& Smyrlis \cite{17}, where the authors also deal with a parametric singular problem and prove a bifurcation-type result. In their problem, the perturbation $f(z,x)$ is ($p-1$)-superlinear in $x\in\RR$ near $+\infty$. So, our present work complements the results of \cite{17}, by considering equations in which the reaction has the competing effects of a singular term and of a $(p-1)$-linear term.

Our approach uses variational tools together with suitable truncation and comparison techniques.

\section{Preliminaries and hypotheses}

Let $X$ be a Banach space and $X^*$ its topological dual. By $\left\langle \cdot,\cdot\right\rangle$ we denote the duality brackets of the pair $(X^*,X)$. Given $\varphi\in C^1(X,\RR)$, we say that $\varphi$ satisfies the ``Cerami condition" (the ``C-condition" for short), if the following property holds:
\begin{center}
``Every sequence $\{u_n\}_{n\geq 1}\subseteq X$ such that
$$\{\varphi(u_n)\}_{n\geq 1}\subseteq\RR\ \mbox{is bounded and}\
\mbox{$(1+||u_n||)\varphi'(u_n)\rightarrow 0$ in $X^*$ as $n\rightarrow\infty$,}$$
admits a strongly convergent subsequence."
\end{center}

Using this notion, we can state the ``mountain pass theorem".
\begin{theorem}\label{th1} {\bf (Mountain pass theorem)}
	Assume that $\varphi\in C^1(X,\RR)$ satisfies the C-condition, $u_0,u_1\in X$, $||u_1-u_0||>\rho>0$,
	$$\max\{\varphi(u_0),\varphi(u_1)\}<\inf\{\varphi(u):||u-u_0||=\rho\}=m_{\rho}$$
	and $c=\inf\limits_{\gamma\in\Gamma}\max\limits_{0\leq t\leq 1}\ \varphi(\gamma(t))$ with $\Gamma=\{\gamma\in C([0,1],X):\gamma(0)=u_0,\gamma(1)=u_1\}$. Then $c\geq m_{\rho}$ and $c$ is a critical value of $\varphi$ (that is, we can find $\hat{u}\in X$ such that $\varphi'(\hat{u})=0$ and $\varphi(\hat{u})=c$).
\end{theorem}

The analysis of problem \eqref{eqp} will involve the Sobolev space $W^{1,p}_0(\Omega)$ and the Banach space $$C^1_0(\overline{\Omega})=\{u\in C^1(\overline{\Omega}):u|_{\partial\Omega}=0\}.$$ We denote by $||\cdot||$ the norm of $W^{1,p}_0(\Omega)$. On account of the Poincar\'e inequality, we have
$$||u||=||Du||_p\ \mbox{for all}\ u\in W^{1,p}_0(\Omega).$$

The space $C^1_0(\overline{\Omega})$ is an ordered Banach space with positive (order) cone $$C_+=\{u\in C^1_0(\overline{\Omega}):u(z)\geq 0\ \mbox{for all}\ z\in\overline{\Omega}\}.$$ This cone has a nonempty interior given by
$${\rm int}\, C_+=\left\{u\in C_+:u(z)>0\ \mbox{for all}\ z\in\Omega,\ \left.\frac{\partial u}{\partial n}\right|_{\partial\Omega}<0\right\}.$$

Here, $n(\cdot)$ denotes the outward unit normal on $\partial\Omega$.

Let $h_1,h_2\in L^{\infty}(\Omega)$. We write $h_1\prec h_2$, if for every compact $K\subseteq\Omega$, we can find $c_K>0$ such that $c_K\leq h_2(z)-h_1(z)$ for almost all $z\in K$. Note that, if $h_1,h_2\in C(\Omega)$ and $h_1(z)<h_2(z)$ for all $z\in\Omega$, then $h_1\prec h_2$.

The next strong comparison principle can be found in Papageorgiou \& Smyrlis \cite[Proposition 4]{17} (see also Giacomoni, Schindler \& Taka\v{c} \cite[Theorem 2.3]{7}).

\begin{prop}\label{prop2}
	If $\hat{\xi}\geq 0,h_1,h_2\in L^{\infty}(\Omega)$, $h_1\prec h_2,u_1\in C_+$ with $u_1(z)>0$ for all $z\in\Omega$, $u_2\in {\rm int}\, C_+$ and
	\begin{eqnarray*}
		&&-\Delta_pu_1(z)+\hat{\xi}u_1(z)^{p-1}-\lambda u_1(z)^{-\gamma}=h_1(z),\\
		&&-\Delta_pu_2(z)+\hat{\xi}u_2(z)^{p-1}-\lambda u_2(z)^{-\gamma}=h_2(z)\ \mbox{for almost all}\ z\in\Omega,
	\end{eqnarray*}
	then $u_2-u_1\in {\rm int}\, C_+.$
\end{prop}

We denote by $A:W^{1,p}_0(\Omega)\rightarrow W^{-1,p'}(\Omega)=W^{1,p}_0(\Omega)^*\left(\frac{1}{p}+\frac{1}{p'}=1\right)$  the nonlinear map defined by
$$\left\langle A(u),h\right\rangle=\int_{\Omega}|Du|^{p-2}(Du,Dh)_{\RR^N}dz\ \mbox{for all}\ u,h\in W^{1,p}_0(\Omega).$$

This map has the following properties (see Motreanu, Motreanu \& Papageorgiou \cite[p. 40]{15}).
\begin{prop}\label{prop3}
	The map $A:W^{1,p}_0(\Omega)\rightarrow W^{-1,p'}(\Omega)$ is bounded (that is, $A$ maps bounded sets to bounded sets), continuous, strictly monotone and of type $(S)_+$, that is, if $u_n\stackrel{w}{\rightarrow}u$ in $W^{1,p}_0(\Omega)$ and $\limsup\limits_{n\rightarrow\infty}\left\langle A(u_n),u_n-u\right\rangle\leq 0$, then $u_n\rightarrow u$ in $W^{1,p}_0(\Omega)$.
\end{prop}

Consider the following nonlinear eigenvalue problem
\begin{equation}\label{eq1}
	-\Delta_pu(z)=\hat{\lambda}|u(z)|^{p-2}u(z)\ \mbox{in}\ \Omega,\ u|_{\partial\Omega}=0.
\end{equation}

We say that $\hat{\lambda}\in\RR$ is an ``eigenvalue" of ($-\Delta_p,W^{1,p}_0(\Omega)$) if problem (\ref{eq1}) admits a nontrivial solution $\hat{u}\in W^{1,p}_0(\Omega)$, known as an ``eigenfunction" corresponding to $\hat{\lambda}$. The nonlinear regularity theory (see Gasinski \& Papageorgiou \cite[pp. 737-738]{3}) implies that $\hat{u}\in C^1_0(\overline{\Omega})$. There is a smallest eigenvalue $\hat{\lambda}_1>0$ with the following properties:
\begin{itemize}
	\item $\hat{\lambda}_1>0$ is isolated (that is, if $\hat{\sigma}(p)$ denotes the spectrum of ($-\Delta_p,W^{1,p}_0(\Omega)$) then we can find $\epsilon>0$ such that $(\hat{\lambda}_1,\hat{\lambda}_1+\epsilon)\cap\hat{\sigma}(p)=0$);
	\item $\hat{\lambda}_1$ is simple (that is, if $\hat{u},\hat{v}\in C^1_0(\overline{\Omega})$ are eigenfunctions corresponding to $\hat{\lambda}_1$, then $\hat{u}=\xi\hat{v}$ for some $\xi\in \RR\backslash\{0\}$);
	\begin{equation}\label{eq2}
		\bullet\hspace{3cm}\hat{\lambda}_1=\inf\left\{\frac{||Du||^p_p}{||u||^p_p}:u\in W^{1,p}_0(\Omega),u\neq 0\right\}.\hspace{4cm}
	\end{equation}
\end{itemize}

It follows from the above properties that the eigenfunctions corresponding to $\hat{\lambda}_1$ do not change sign. We denote by $\hat{u}_1$ the positive, $L^p$-normalized (that is, $||\hat{u}_1||_p=1$) eigenfunction corresponding to $\hat{\lambda}_1>0$. From the nonlinear maximum principle (see, for example, Gasinski \& Papageorgiou \cite[p. 738]{3}), we have $\hat{u}_1\in {\rm int}\, C_+$. Any eigenfunction corresponding to an eigenvalue $\hat{\lambda}\neq\hat{\lambda}_1$, is nodal (that is, sign-changing). More details about the spectrum of $(-\Delta_p,W^{1,p}_0(\Omega))$ can be found in \cite{3, 15}.

We can also consider a weighted version of the eigenvalue problem (\ref{eq1}). So, let $m\in L^{\infty}(\Omega)$, $m(z)\geq 0$ for almost all $z\in\Omega,\ m\neq 0$. We consider the following nonlinear eigenvalue problem:
\begin{equation}\label{eq3}
	-\Delta_pu(z)=\tilde{\lambda}m(z)|u(z)|^{p-2}u(z)\ \mbox{in}\ \Omega,\ u|_{\partial\Omega}=0.
\end{equation}

This problem has the same properties as (\ref{eq1}). So, there is a smallest eigenvalue $\tilde{\lambda}_1(m)>0$ which is isolated, simple and admits the following variational characterization
$$\tilde{\lambda}_1(m)=\inf\left\{\frac{||Du||^p_p}{\int_{\Omega}m(z)|u|^pdz}:u\in W^{1,p}_0(\Omega),u\neq 0\right\}.$$

Also the eigenfunctions corresponding to $\tilde{\lambda}_1(m)$ have a fixed sign and we denote by $\tilde{u}_1(m)$  the positive, $L^p$-normalized eigenfunction. We have $\tilde{u}_1(m)\in {\rm int}\, C_+$. These properties lead to the following monotonicity property of the map $m\mapsto\tilde{\lambda}_1(m)$.

\begin{prop}\label{prop4}
	If $m_1,m_2\in L^{\infty}(\Omega),0\leq m_1(z)\leq m_2(z)$ for almost all $z\in\Omega$ and both inequalities are strict on the sets of positive measure, then $\tilde{\lambda}_1(m_2)<\tilde{\lambda}_1(m_1)$.
\end{prop}

Given $x\in\RR$, we set $x^{\pm}=\max\{\pm x, 0\}$. Then for $u\in W^{1,p}_0(\Omega)$, we set $u^{\pm}(\cdot)=u(\cdot)^{\pm}$. We know that
$$u^{\pm}\in W^{1,p}_0(\Omega),\ |u|=u^++u^-,\ u=u^+-u^-.$$

If $g:\Omega\times\RR$ is a measurable function (for example, a Carath\'eodory function) then by $N_g(\cdot)$ we denote the Nemytski map corresponding to $g(\cdot,\cdot)$ defined by
$$N_g(u)(\cdot)=g(\cdot,u(\cdot))\ \mbox{for all}\ u\in W^{1,p}_0(\Omega).$$

Given $v,u\in W^{1,p}_0(\Omega)$ with $v\leq u$, we define the order interval $[v,u]$ by
$$[v,u]=\{y\in W^{1,p}_0(\Omega):v(z)\leq y(z)\leq u(z)\ \mbox{for almost all}\ z\in\Omega\}.$$

The hypotheses on the perturbation $f(z,x)$ are the following:

\smallskip
$H(f):$ $f:\Omega\times\RR\leftarrow\RR$ is a Carath\'eodory function such that $f(z,0)=0$ for almost all $z\in\Omega$ and
\begin{itemize}
	\item[(i)] for every $\rho>0$, there exists $a_{\rho}\in L^{\infty}(\Omega)$ such that
	$$|f(z,x)|\leq a_{\rho}(z)\ \mbox{for almost all}\ z\in\Omega,\ \mbox{and all}\ 0\leq x\leq\rho;$$
	\item[(ii)] $\hat{\lambda}_1<\eta\leq\liminf\limits_{x\rightarrow+\infty}\frac{f(z,x)}{x^{p-1}}\leq\limsup\limits_{x\rightarrow+\infty}\frac{f(z,x)}{x^{p-1}}\leq\hat{\eta}$ uniformly for almost all $z\in\Omega;$
	\item[(iii)] there exists a function $w\in C^1(\overline{\Omega})$ such that
	$$w(z)\geq c_0>0\ \mbox{for all}\ z\in\overline{\Omega},\ \Delta_pw\in L^{\infty}(\Omega)\ \mbox{with}\ \Delta_pw(z)\leq 0\ \mbox{for almost all}\ z\in\Omega,$$
	and for every compact $K\subseteq\Omega$ we can find $c_K>0$ such that
	$$w(z)^{-\gamma}+f(z,w(z))\leq-c_K<0\ \mbox{for almost all}\ z\in K;$$
	\item[(iv)] there exists $\delta_0\in(0,c_0)$ such that for every compact $K\subseteq\Omega$
	$$f(z,x)\geq\hat{c}_K>0\ \mbox{for almost all}\ z\in K,\ \mbox{and all}\ x\in\left(0,\delta_0\right];$$
	\item[(v)] for every $\rho>0$, there exists $\hat{\xi}_{\rho}>0$ such that for almost all $z\in\Omega$ the function
	$$x\mapsto f(z,x)+\hat{\xi}_{\rho}x^{p-1}$$
	is nondecreasing on $[0,\rho]$.
\end{itemize}

\begin{remark}
	Since we are looking for positive solutions and all the above hypotheses concern the positive semiaxis $\RR_+=\left[0,+\infty\right)$, we may assume without any loss of generality that
	\begin{equation}\label{eq4}
		f(z,x)=0\ \mbox{for almost all}\ z\in\Omega,\ \mbox{and all}\ x\leq 0.
	\end{equation}
\end{remark}

Hypothesis $H(f)(iii)$ implies that asymptotically at $+\infty$ we have uniform nonresonance with respect to the principal eigenvalue $\hat{\lambda}_1>0$ of $(-\Delta_p,W^{1,p}_0(\Omega))$. The resonant case was recently examined  for nonparametric singular Dirichlet problems by Papageorgiou, R\u{a}dulescu \& Repov\v{s} \cite{16}.

\begin{ex}
The following functions satisfy hypotheses $H(f)$. For the sake of simplicity we drop the $z$-dependence:
$$f(x)=\left\{\begin{array}{ll}
	x^{\tau-1}-3x^{\vartheta-1}&\mbox{if}\ 0\leq x\leq 1\\
	\eta x^{p-1}-(\eta+2)x^{q-1}&\mbox{if}\ 1<x
\end{array}\right\}\ (\mbox{see (\ref{eq4})})$$
with $1<\tau<\vartheta$, $1<q<p$ and $\eta>\hat{\lambda}_1$; and

\textcolor{black}{
$$f(x)=\left\{\begin{array}{ll}
	2\sin(2\pi x)&\mbox{if}\ 0\leq x\leq 1\\
	\eta (x^{p-1}-x^{q-1})&\mbox{if}\ 1<x
\end{array}\right.$$
with $\eta>\hat{\lambda}_1$, $1<q<p$. }
\end{ex}

\section{A purely singular problem}

In this section we deal with the following purely singular parametric problem:
\begin{equation}\tag{$Au_{\lambda}$}\label{eqa}
	\left\{\begin{array}{l}
		-\Delta_pu(z)=\lambda u(z)^{-\gamma}\ \mbox{in}\ \Omega\\
		u|_{\partial\Omega}=0,\ u>0,\ \lambda>0,\ 0<\gamma<1.
	\end{array}\right\}
\end{equation}

The next proposition establishes the existence and $\lambda$-dependence of the positive solutions for problem \eqref{eqa}.
\begin{prop}\label{prop5}
	For every $\lambda>0$ problem \eqref{eqa} admits a unique solution $\tilde{u}_{\lambda}\in {\rm int}\, C_+$, the map $\lambda\mapsto\tilde{u}_{\lambda}$ is nondecreasing from $(0,\infty)$ into $C^1_0(\overline{\Omega})$ (that is, if $0<\vartheta<\lambda$, then $\tilde{u}_{\vartheta}\leq\tilde{u}_{\lambda}$) and $||\tilde{u}_{\lambda}||_{C^1_0(\overline{\Omega})}\rightarrow 0$ as $\lambda\rightarrow 0^+$.
\end{prop}
\begin{proof}
	The existence of a unique solution $\tilde{u}_{\lambda}\in {\rm int}\, C_+$ follows from Proposition 5 of Papageorgiou \& Smyrlis \cite{17}.
	
	Let $0<\vartheta<\lambda$ and let $\tilde{u}_{\vartheta},\tilde{u}_{\lambda}\in {\rm int}\, C_+$ be the corresponding unique solutions of problem \eqref{eqa}. Evidently, $\tilde{u}^{p'}_{\vartheta}\in {\rm int}\, C_+\left(\frac{1}{p}+\frac{1}{p'}=1\right)$ and so by Proposition 2.1 of Marano \& Papageorgiou \cite{14}, we can find $c_1>0$ such that
	\begin{eqnarray*}
		&&\hat{u}_1\leq c_1\tilde{u}^{p'}_{\vartheta},\\
		&\Rightarrow&\hat{u}_1^{1/p'}\leq c_1^{1/p'}\tilde{u}_{\vartheta},\\
		&\Rightarrow&\tilde{u}^{-\gamma}_{\vartheta}\leq c_2\hat{u}_1^{-\gamma/p'}\ \mbox{for some}\ c_2>0.
	\end{eqnarray*}
	
	 Lemma of Lazer \& McKenna \cite[p. 726]{12}, implies that $\hat{u}_1^{-\gamma/p'}\in L^{p'}(\Omega)$. Therefore $\tilde{u}_{\vartheta}^{-\gamma}\in L^{p'}(\Omega)$. We introduce the Carath\'eodory function $g_{\lambda}(z,x)$ defined by
	\begin{equation}\label{eq5}
		g_{\lambda}(z,x)=\left\{\begin{array}{ll}
			\lambda\tilde{u}_{\vartheta}^{-\gamma}&\mbox{if}\ x\leq\tilde{u}_{\vartheta}(z)\\
			\lambda x^{-\gamma}&\mbox{if}\ \tilde{u}_{\vartheta}(z)<x.
		\end{array}\right.
	\end{equation}
	
	We set $G_{\lambda}(z,x)=\int^x_0g_{\lambda}(z,s)ds$ and consider the functional $\hat{\psi}_{\lambda}:W^{1,p}_0(\Omega)\rightarrow\RR$ defined by
	$$\hat{\psi}_{\lambda}(u)=\frac{1}{p}||Du||^p_p-\int_{\Omega}G_{\lambda}(z,u)dz\ \mbox{for all}\ u\in W^{1,p}_0(\Omega).$$
	
	Proposition 3 of Papageorgiou \& Smyrlis \cite{17} implies that $\hat{\psi}_{\lambda}\in C^1(W^{1,p}_0(\Omega))$. From (\ref{eq5}) and since $\tilde{u}_{\vartheta}^{-\gamma}\in L^{p'}(\Omega)$ it follows that $\hat{\psi}_{\lambda}(\cdot)$ is coercive. Also, via the Sobolev embedding theorem, we see that $\hat{\psi}_{\lambda}(\cdot)$ is sequentially weakly lower semicontinuous. So, by the Weierstrass-Tonelli theorem, we can find $\bar{u}_{\lambda}\in W^{1,p}_0(\Omega)$ such that
		\begin{eqnarray}\label{eq6}
			&&\hat{\psi}_{\lambda}(\bar{u}_{\lambda})=\inf\{\hat{\psi}_{\lambda}(u):u\in W^{1,p}_0(\Omega)\},\nonumber\\
			&\Rightarrow&\hat{\psi}'_{\lambda}(\bar{u}_{\lambda})=0,\nonumber\\
			&\Rightarrow&\left\langle A(\bar{u}_{\lambda}),h\right\rangle=\int_{\Omega}g_{\lambda}(z,\bar{u}_{\lambda})hdz\ \mbox{for all}\ h\in W^{1,p}_0(\Omega).
		\end{eqnarray}
		
		In (\ref{eq3}) we choose $h=(\tilde{u}_{\vartheta}-\bar{u}_{\lambda})^+\in W^{1,p}_0(\Omega)$. We have
		\begin{eqnarray}\label{eq7}
			\left\langle A(\bar{u}_{\lambda}),(\tilde{u}_{\vartheta}-\bar{u}_{\lambda})^+\right\rangle&=&\int_{\Omega}\lambda\tilde{u}_{\vartheta}^{-\gamma}(\tilde{u}_{\vartheta}-\bar{u}_{\lambda})^+dz\ (\mbox{see (\ref{eq5})})\nonumber\\
			&\geq&\int_{\Omega}\vartheta\tilde{u}_{\vartheta}^{-\gamma}(\tilde{u}_{\vartheta}-\bar{u}_{\lambda})dz\ (\mbox{since}\ \vartheta<\lambda)\nonumber\\
			&=&\left\langle A(\tilde{u}_{\vartheta}),(\tilde{u}_{\vartheta}-\bar{u}_{\lambda})^+\right\rangle,\nonumber\\
			\Rightarrow\tilde{u}_{\vartheta}\leq\bar{u}_{\lambda}.
		\end{eqnarray}
		
		From (\ref{eq5}), (\ref{eq6}), (\ref{eq7}), we have
		\begin{eqnarray*}
			&&-\Delta_p\bar{u}_{\lambda}(z)=\lambda\bar{u}_{\lambda}(z)^{-\gamma}\ \mbox{for almost all}\ z\in\Omega,\left.\bar{u}_{\lambda}\right|_{\partial\Omega}=0,\\
			&\Rightarrow&\bar{u}_{\lambda}=\tilde{u}_{\lambda},\\
			&\Rightarrow&\tilde{u}_{\vartheta}\leq\tilde{u}_{\lambda}\ (\mbox{see (\ref{eq7})}).
		\end{eqnarray*}
		
		Therefore the map $\lambda\mapsto\tilde{u}_{\lambda}$ is nondecreasing from $(0,+\infty)$ into $C^1_0(\overline{\Omega})$.
		
		We have
		$$\left\langle A(\tilde{u}_{\lambda}),h\right\rangle=\int_{\Omega}\lambda\tilde{u}_{\lambda}^{-\gamma}hdz\ \mbox{for all}\ h\in W^{1,p}_0(\Omega).$$
		
		Choosing $h=\tilde{u}_{\lambda}\in W^{1,p}_0(\Omega)$, we obtain
		\begin{eqnarray}\label{eq8}
			&&||D\tilde{u}_{\lambda}||^p_p=\lambda\int_{\Omega}\tilde{u}_{\lambda}^{1-\gamma}dz\leq \lambda c_3||\tilde{u}_{\lambda}||_p\ \mbox{for some}\ c_3>0\nonumber\\
			&&(\mbox{see Theorem 13.17 of Hewitt \& Stromberg \cite[p. 196]{9}}),\nonumber\\
			&\Rightarrow&\{\tilde{u}_{\lambda}\}_{\lambda\in\left(0,1\right]}\subseteq W^{1,p}_0(\Omega)\ \mbox{is bounded and }||\tilde{u}_{\lambda}||\rightarrow 0\ \mbox{as}\ \lambda\rightarrow 0^+.
		\end{eqnarray}
		
		As in the first part of the proof, using Proposition 2.1 of Marano \& Papageorgiou \cite{14}, we show that $\tilde{u}_{\lambda}^{-\gamma}\in L^r(\Omega)$ for $r>N$. Then Proposition 1.3 of Guedda \& V\'eron \cite{8} implies that
		\begin{equation}\label{eq9}
			\tilde{u}_{\lambda}\in L^{\infty}(\Omega)\ \mbox{and}\ ||\tilde{u}_{\lambda}||_{\infty}\leq c_4\ \mbox{for some}\ c_4>0,\ \mbox{and all}\ 0<\lambda\leq 1.
		\end{equation}
		
		Let $k_{\lambda}=\lambda\tilde{u}^{-\gamma}_{\lambda}\in L^r(\Omega),\lambda\in\left(0,1\right]$ and consider the following linear Dirichlet problem
		\begin{equation}\label{eq10}
			-\Delta v(z)=k_{\lambda}(z)\ \mbox{in}\ \Omega,\ v|_{\partial\Omega}=0,\ 0<\lambda\leq 1.
		\end{equation}
		
		Standard existence and regularity theory (see, for example, Struwe \cite[p. 218]{19}), implies that problem (\ref{eq10}) has a unique solution $v_{\lambda}(\cdot)$ such that
		$$v_{\lambda}\in W^{2,r}(\Omega)\subseteq C^{1,\alpha}_{0}(\overline{\Omega})=C^{1,\alpha}(\overline{\Omega})\cap C^1_0(\overline{\Omega}),\ ||v_{\lambda}||_{C^{1,\alpha}_0(\overline{\Omega})}\leq c_5$$
		for some $c_5>0$, all $\lambda\in\left(0,1\right]$, and with $\alpha=1-\frac{N}{r}\in(0,1)$ (recall that $r>N$). Let $\beta_{\lambda}(z)=Dv_{\lambda}(z)$. Then $\beta_{\lambda}\in C^{0,\alpha}(\overline{\Omega})$ for every $\lambda\in\left(0,1\right]$. We have
		$$-{\rm div}\,[|D\tilde{u}_{\lambda}|^{p-2}D\tilde{u}_{\lambda}-\beta_{\lambda}]=0\ \mbox{in}\ \Omega,\left.\ \tilde{u}_{\lambda}\right|_{\partial\Omega}=0\ (\mbox{since}\ \tilde{u}_{\lambda}\ \mbox{solves}\ \eqref{eqa}).$$
		
		Then Theorem 1 of Lieberman \cite{13} (see also Corollary 1.1 of Guedda \& V\'eron \cite{8}) and (\ref{eq9}), imply that we can find $s\in(0,1)$ and $c_6>0$ such that
		$$\tilde{u}_{\lambda}\in C^{1,s}_0(\overline{\Omega})\cap {\rm int}\, C_+,\ ||\tilde{u}_{\lambda}||_{C^{1,s}_0(\overline{\Omega})}\leq c_6\ \mbox{for all}\ \lambda\in\left(0,1\right].$$
		
		Finally, the compact embedding of $C^{1,s}_0(\overline{\Omega})$ into $C^1_0(\overline{\Omega})$ and (\ref{eq8}) imply that
		$$||\tilde{u}_{\lambda}||_{C^1_0(\overline{\Omega})}\rightarrow 0\ \mbox{as}\ \lambda\rightarrow 0^+.$$
This completes the proof.
\end{proof}

\section{Bifurcation-type theorem}

Let $$\mathcal{L}=\{\lambda>0:\ \mbox{problem \eqref{eqp} admits a positive solution}\}$$ \begin{center}
$S_{\lambda}=\mbox{the set of positive solutions for problem \eqref{eqp}}$.
  \end{center}
\begin{prop}\label{prop6}
If hypotheses $H(f)$ hold, then $\mathcal{L}\neq\emptyset$.
\end{prop}
\begin{proof}
Using Proposition \ref{prop5}, we can find $\lambda_0\in\left(0,1\right]$ such that
\begin{equation}\label{eq11}
\tilde{u}_{\lambda}(z)\in\left(0,\delta_0\right]\ \mbox{for all}\ z\in\Omega,\ \mbox{all}\ \lambda\in\left(0,\lambda_0\right].
\end{equation}

Here, $\delta_0>0$ is as postulated by hypothesis $H(f)(iv)$.

We fix $\lambda\in\left(0,\lambda_0\right]$ and we consider the following truncation of the reaction in problem \eqref{eqp}:
\begin{eqnarray}\label{eq12}
\hat{k}_{\lambda}(z,x)=\left\{\begin{array}{ll}
    \lambda\hat{u}_{\lambda}(z)^{-\gamma}+f(z,\tilde{u}_{\lambda}(z))&\mbox{if}\ x<\tilde{u}_{\lambda}(z)\\
    \lambda x^{-\gamma}+f(z,x)&\mbox{if}\ \tilde{u}_{\lambda}\leq x\leq w(z)\\
    \lambda w(z)^{-\gamma}+f(z,w(z))&\mbox{if}\ w(z)<x
\end{array}
\right.
\end{eqnarray}
(recall that $\delta_0<c_0\leq w(z)$ for all $z\in\overline{\Omega}$). This is a Carath\'eodory function. We set $\hat{K}_{\lambda}(z,x)=\int^x_0\hat{k}_{\lambda}(z,s)ds$ and consider the function $\hat{\varphi}_{\lambda}:W^{1,p}_0(\Omega)\rightarrow \RR$ defined by
$$\hat{\varphi}_{\lambda}(u)=\frac{1}{p}||Du||^p_p-\int_{\Omega}\hat{K}_{\lambda}(z,u)dz\ \mbox{for all}\ u\in W^{1,p}_0(\Omega).$$

As before, we have $\hat{\varphi}_{\lambda}\in C^1(W^{1,p}_0(\Omega))$. Also,  it follows from (\ref{eq12}) that
$$\hat{\varphi}(\cdot)\ \mbox{is coercive}.$$

In addition, we have that
$$\hat{\varphi}_{\lambda}(\cdot)\ \mbox{is sequentially lower semicontinuous}.$$

Therefore, we can find $\hat{u}_{\lambda}\in W^{1,p}_0(\Omega)$ such that
\begin{eqnarray}\label{eq13}
&&\hat{\varphi}_{\lambda}(\hat{u}_{\lambda})=\inf[\hat{\varphi}_{\lambda}(u):u\in W^{1,p}_0(\Omega)],\nonumber\\
&\Rightarrow &\hat{\varphi}'_{\lambda}(\hat{u}_{\lambda})=0,\nonumber\\
&\Rightarrow &\left\langle A(\hat{u}_{\lambda}),h\right\rangle=\int_{\Omega}\hat{k}_{\lambda}(z,\hat{u}_{\lambda})hdz\ \mbox{for all}\ h\in W^{1,p}_0(\Omega).
\end{eqnarray}

In (\ref{eq13}) we choose $h=(\tilde{u}_{\lambda}-\hat{u}_{\lambda})^+\in W^{1,p}_0(\Omega)$. Then
\begin{eqnarray*}
	\left\langle A(\hat{u}_{\lambda}),(\tilde{u}_{\lambda}-\hat{u}_{\lambda})^+\right\rangle&=&\int_{\Omega}[\lambda\tilde{u}_{\lambda}^{-\gamma}+f(z,\tilde{u}_{\lambda})](\tilde{u}_{\lambda}-\hat{u}_{\lambda})^+dz\ (\mbox{see (\ref{eq12})})\\
	&\geq&\int_{\Omega}\lambda\tilde{u}_{\lambda}^{-\gamma}(\tilde{u}_{\lambda}-\hat{u}_{\lambda})^+dz\\
	&&(\mbox{see (\ref{eq11}) and hypothesis H(f)(iv)})\\
	&=&\left\langle A(\tilde{u}_{\lambda}),(\tilde{u}_{\lambda}-\hat{u}_{\lambda})^+\right\rangle\ (\mbox{see Proposition \ref{prop5}}),\\
	\Rightarrow\tilde{u}_{\lambda}\leq\hat{u}_{\lambda}.&
\end{eqnarray*}

Next, we choose   $h=(\hat{u}_{\lambda}-w)^+\in W^{1,p}_0(\Omega)$ in (\ref{eq13}). Then
\begin{eqnarray*}
	\left\langle A(\hat{u}_{\lambda}),(\hat{u}_{\lambda}-w)^+\right\rangle&=&\int_{\Omega}[\lambda w^{-\gamma}+f(z,w)](\hat{u}_{\lambda}-w)^+dz\ (\mbox{see (\ref{eq12})})\\
	&\leq&\left\langle A(w),(\hat{u}_{\lambda}-w)^+\right\rangle
\end{eqnarray*}
(see hypothesis $H(f)(iii)$ and use the nonlinear Green identity, see \cite[p. 211]{3})
$$\Rightarrow\tilde{u}_{\lambda}\leq w.$$

So, we have proved that
\begin{equation}\label{eq14}
	\hat{u}_{\lambda}\in[\tilde{u}_{\lambda},w].
\end{equation}

Using (\ref{eq14}) and (\ref{eq12}), equation (\ref{eq13}) becomes
\begin{eqnarray}\label{eq15}
	&&\left\langle A(\hat{u}_{\lambda}),h\right\rangle=\int_{\Omega}[\lambda\hat{u}_{\lambda}^{-\gamma}+f(z,\hat{u}_{\lambda})]hdz\ \mbox{for all}\ h\in W^{1,p}_0(\Omega),\nonumber\\
	&\Rightarrow&-\Delta_p\hat{u}_{\lambda}(z)=\lambda\hat{u}_{\lambda}(z)^{-\gamma}+f(z,\hat{u}_{\lambda}(z))\ \mbox{for almost all}\ z\in\Omega,\left.\hat{u}_{\lambda}\right|_{\partial\Omega}=0.
\end{eqnarray}

From (\ref{eq14}), (\ref{eq15}) and Theorem 1 of Lieberman \cite{13}, we infer that
\begin{eqnarray*}
	&&\hat{u}_{\lambda}\in[\tilde{u}_{\lambda},w]\cap {\rm int}\, C_+,\\
	&\Rightarrow&\lambda\in\mathcal{L},\hat{u}_{\lambda}\in S_{\lambda}.
\end{eqnarray*}
This completes the proof.
\end{proof}

A byproduct of the above proof is the following corollary.
\begin{corollary}\label{cor7}
	If hypotheses $H(f)$ hold, then $S_{\lambda}\subseteq {\rm int}\, C_+$ for all $\lambda>0$.
\end{corollary}

The next proposition shows that $\mathcal{L}$ is an interval.
\begin{prop}\label{prop8}
	If hypotheses $H(f)$ hold, $\lambda\in\mathcal{L}$ and $\vartheta\in(0,\lambda)$, then $\vartheta\in\mathcal{L}.$
\end{prop}
\begin{proof}
	Since $\lambda\in\mathcal{L}$, we can find $u_{\lambda}\in S_{\lambda}\subseteq {\rm int}\, C_+$. Proposition \ref{prop5} implies that we can find $\tau\in[0,\lambda_0]$ (see (\ref{eq11})) such that
	$$\tau<\vartheta\ \mbox{and}\ \tilde{u}_{\tau}\leq u_{\lambda}.$$
	
	We introduce the Carath\'eodory function $e(z,x)$ defined by
	\begin{equation}\label{eq16}
		e_{\vartheta}(z,x)=\left\{\begin{array}{ll}
			\vartheta\tilde{u}_{\tau}(z)^{-\gamma}+f(z,\tilde{u}_{\tau}(z))&\mbox{if}\ x<\tilde{u}_{\tau}(z)\\
			\vartheta x^{-\gamma}+f(z,x)&\mbox{if}\ \tilde{u}_{\tau}(z)\leq x\leq u_{\lambda}(z)\\
			\vartheta u_{\lambda}(z)^{-\gamma}+f(z,u_{\lambda}(z))&\mbox{if}\ u_{\lambda}(z)<x.
		\end{array}\right.
	\end{equation}
	
	We set $E_{\vartheta}(z,x)=\int^x_0 e_{\vartheta}(z,s)ds$ and consider the functional $\hat{\psi}_{\vartheta}:W^{1,p}_0(\Omega)\rightarrow\RR$ defined by
	$$\hat{\psi}_{\vartheta}(u)=\frac{1}{p}||Du||^p_p-\int_{\Omega}E_{\vartheta}(z,u)dz\ \mbox{for all}\ u\in W^{1,p}_0(\Omega).$$
	
	We know that $\hat{\psi}_{\vartheta}\in C^1(W^{1,p}_0(\Omega))$. Moreover, $\hat{\psi}_{\vartheta}$ is coercive (see (\ref{eq16})) and sequentially weakly lower semicontinuous. So, we can find $u_{\vartheta}\in W^{1,p}_0(\Omega)$ such that
	\begin{eqnarray}\label{eq17}
		&&\hat{\psi}_{\vartheta}(u_{\vartheta})=\inf\{\hat{\psi}_{\vartheta}(u):u\in W^{1,p}_0(\Omega)\},\nonumber\\
		&\Rightarrow&\hat{\psi}'_{\vartheta}(u_{\vartheta})=0,\nonumber\\
		&\Rightarrow&\left\langle A(u_{\vartheta}),h\right\rangle=\int_{\Omega}e_{\vartheta}(z,u_{\vartheta})hdz\ \mbox{for all}\ h\in W^{1,p}_0(\Omega).
	\end{eqnarray}
	
	In (\ref{eq17})  we first choose $h=(\tilde{u}_{\tau}-u_{\vartheta})^+\in W^{1,p}_0(\Omega)$. Then
	\begin{eqnarray*}
		\left\langle A(u_{\vartheta}),(\tilde{u}_{\tau}-u_{\vartheta})^+\right\rangle&=&\int_{\Omega}[\vartheta\tilde{u}_{\tau}^{-\gamma}+f(z,\tilde{u}_{\tau})](\tilde{u}_{\tau}-u_{\vartheta})^+dz\ (\mbox{see (\ref{eq16})})\\
		&\geq&\int_{\Omega}\vartheta\tilde{u}_{\tau}^{-\gamma}(\tilde{u}_{\tau}-u_{\vartheta})^+dz\\
		&&(\mbox{since}\ \tau\leq\lambda_0,\ \mbox{see (\ref{eq11}) and hypothesis}\ H(f)(iv))\\
		&\geq&\int_{\Omega}\tau\tilde{u}_{\tau}^{-\gamma}(\tilde{u}_{\tau}-u_{\vartheta})^+dz\ (\mbox{recall that}\ \tau<\vartheta)\\
		&=&\left\langle A(u_{\tau}),(\tilde{u}_{\tau}-u_{\vartheta})^+\right\rangle\ (\mbox{see Proposition \ref{eq5}}),\\
		\Rightarrow\tilde{u}_{\tau}\leq u_{\vartheta}.&&
	\end{eqnarray*}
	
	Next, in (\ref{eq17}) we choose $h=(u_{\vartheta}-u_{\lambda})^+\in W^{1,p}_0(\Omega)$. Then
	\begin{eqnarray*}
		\left\langle A(u_{\vartheta}),(u_{\vartheta}-u_{\lambda})^+\right\rangle&=&\int_{\Omega}[\vartheta u_{\lambda}^{-\gamma}+f(z,u_{\lambda})](u_{\vartheta}-u_{\lambda})^+dz\ (\mbox{see (\ref{eq16})})\\
		&\leq&\int_{\Omega}[\lambda u_{\lambda}^{-\gamma}+f(z,u_{\lambda})](u_{\vartheta}-u_{\lambda})^+dz\ (\mbox{since}\ \vartheta<\lambda)\\
		&=&\left\langle A(u_{\lambda}),(u_{\vartheta}-u_{\lambda})^+\right\rangle\ (\mbox{since}\ u_{\lambda}\in S_{\lambda}),\\
		\Rightarrow u_{\vartheta}\leq u_{\lambda}.&&
	\end{eqnarray*}
	
	So, we have proved that
	\begin{equation}\label{eq18}
		u_{\vartheta}\in[\tilde{u}_{\tau},u_{\lambda}].
	\end{equation}
	
	It follows from (\ref{eq16}), (\ref{eq17}) and (\ref{eq18}) that
	$$\vartheta\in\mathcal{L}\ \mbox{and}\ u_{\vartheta}\in S_{\vartheta}\subseteq {\rm int}\, C_+.$$
The proof is now complete.
\end{proof}

An interesting byproduct of the above proof is the following result.
\begin{corollary}\label{cor9}
	If hypotheses $H(f)$ hold, $\lambda\in\mathcal{L},u_{\lambda}\in S_{\lambda}\subseteq {\rm int}\, C_+$, and $\vartheta<\lambda$, then $\vartheta\in\mathcal{L}$ and we can find $u_{\vartheta}\in S_{\vartheta}\subseteq {\rm int}\, C_+$ such that $u_{\vartheta}\leq u_{\lambda}$.
\end{corollary}

In fact, we can improve the above result as follows.
\begin{prop}\label{prop10}
	If hypotheses $H(f)$ hold, $\lambda\in\mathcal{L},u_{\lambda}\in S_{\lambda}\subseteq {\rm int}\, C_+$, and $\vartheta<\lambda$, then $\vartheta\in\mathcal{L}$ and we can find $u_{\vartheta}\in S_{\vartheta}\subseteq {\rm int}\, C_+$ such that $u_{\lambda}-u_{\vartheta}\in {\rm int}\, C_+$.
\end{prop}
\begin{proof}
	From Corollary \ref{cor9} we know that $\vartheta\in\mathcal{L}$ and we can find $u_{\vartheta}\in S_{\vartheta}\subseteq {\rm int}\, C_+$ such that
	\begin{equation}\label{eq19}
		u_{\vartheta}\leq u_{\lambda}.
	\end{equation}
	
	Let $\rho=||u_{\lambda}||_{\infty}$ and let $\hat{\xi}_{\rho}>0$ be as postulated by hypothesis $H(f)(v)$. Then
	\begin{eqnarray*}
		&&-\Delta_pu_{\vartheta}+\hat{\xi}_pu_{\vartheta}^{p-1}-\lambda u_{\vartheta}^{-\gamma}\\
		&=&-(\lambda-\vartheta)u_{\vartheta}^{-\gamma}+f(z,u_{\vartheta})+\hat{\xi}_{\rho}u_{\vartheta}^{p-1}\\
		&\leq&f(z,u_{\lambda})+\hat{\xi}_{\rho}u_{\lambda}^{p-1}\ (\mbox{recall that}\ \vartheta<\lambda\ \mbox{and see (\ref{eq19}) and hypothesis}\ H(f)(v))\\
		&=&-\Delta_pu_{\lambda}+\hat{\xi}_{\rho}u_{\lambda}^{p-1}-\lambda u_{\lambda}^{-\gamma}\ (\mbox{since}\ u_{\lambda}\in S_{\lambda}).
	\end{eqnarray*}
	
	We set
	\begin{eqnarray*}
		&&h_1(z)=f(z,u_{\vartheta}(z))+\hat{\xi}_{\rho}u_{\vartheta}(z)^{p-1}-(\lambda-\vartheta)u_{\vartheta}(z)^{-\gamma}\\
		&&h_2(z)=f(z,u_{\lambda}(z))+\hat{\xi}_{\rho}u_{\lambda}(z)^{p-1}.
	\end{eqnarray*}
	
	We have
	$$h_2(z)-h_1(z)\geq(\lambda-\vartheta)u_{\vartheta}(z)^{-\gamma}\geq(\lambda-\vartheta)\rho^{-\gamma}\ \mbox{for almost all}\ z\in\Omega$$
	(see (\ref{eq19}) and hypotheses $H(f)(v)$).
	
	We can apply Proposition \ref{prop2} and conclude that
	$$u_{\lambda}-u_{\vartheta}\in {\rm int}\, C_+.$$
The proof is now complete.
\end{proof}

Denote $\lambda^*=\sup\mathcal{L}.$
\begin{prop}\label{prop11}
	If hypotheses $h(f)$ hold, then $\lambda^*<+\infty$.
\end{prop}
\begin{proof}
	Let $\epsilon>0$ be such that $\hat{\lambda}_1+\epsilon<\eta$ (see hypothesis $H(f)(ii)$). We can find $M>0$ such that
	\begin{equation}\label{eq20}
		f(z,x)\geq[\hat{\lambda}_1+\epsilon]x^{p-1}\ \mbox{for almost all}\ z\in\Omega,\ \mbox{and all}\ x\geq M.
	\end{equation}
	
	Also, hypothesis $H(f)(i)$ implies that we can find large enough $\tilde{\lambda}>0$ such that
	\begin{equation}\label{eq21}
		\tilde{\lambda}M^{-\gamma}+f(z,x)\geq[\hat{\lambda}_1+\epsilon]M^{p-1}\ \mbox{for almost all}\ z\in\Omega,\ \mbox{and all}\ 0\leq x\leq M.
	\end{equation}
	
	It follows from (\ref{eq20}) and (\ref{eq21}) that
	\begin{equation}\label{eq22}
		\tilde{\lambda}x^{-\gamma}+f(z,x)\geq[\hat{\lambda}_1+\epsilon]x^{p-1}\ \mbox{for almost all}\ z\in\Omega,\ \mbox{and all}\ x\geq 0.
	\end{equation}
	
	Let $\lambda>\tilde{\lambda}$ and suppose that $\lambda\in\mathcal{L}$. Then we can find $u_{\lambda}\in S_{\lambda}\subseteq {\rm int}\, C_+$. We have
	\begin{equation}\label{eq23}
		-\Delta_pu_{\lambda}=\lambda u_{\lambda}^{-\gamma}+f(z,u_{\lambda})>\tilde{\lambda}u_{\lambda}^{-\gamma}+f(z,u_{\lambda})\geq[\hat{\lambda}_1+\epsilon]u_{\lambda}^{p-1}\ \mbox{for a.a.}\ z\in\Omega\ (\mbox{see (\ref{eq22})}).
	\end{equation}
	
	Since $u_{\lambda}\in {\rm int}\, C_+$, we can find $t\in(0,1)$ so small that
	\begin{equation}\label{eq24}
		\hat{y}_1=t\hat{u}_1\leq u_{\lambda}
	\end{equation}
	(see Proposition 2.1 of Marano \& Papageorgiou \cite{14}). We have
	\begin{equation}\label{eq25}
		-\Delta_p\hat{y}_1=\hat{\lambda}_1\hat{y}_1^{p-1}<[\hat{\lambda}_1+\epsilon]\hat{y}_1^{p-1}\ \mbox{for almost all}\ z\in\Omega.
	\end{equation}
	
	Using (\ref{eq24}), we can define the Carath\'eodory function $\beta(z,x)$ as follows
	\begin{eqnarray}\label{eq26}
		\beta(z,x)=\left\{
		  \begin{array}{lll}
			 & [\hat{\lambda}_1 + \epsilon]\hat{y}_1(z)^{p-1}&  \mbox{if}\ x<\hat{y}_1(z)\\
			 & [\hat{\lambda}_1 + \epsilon]x^{p-1}           &  \mbox{if}\ \hat{y}_1(z)\leq x\leq u_{\lambda}(z)\\
			 & [\hat{\lambda}_1 + \epsilon]u_{\lambda}(z)^{p-1}&  \mbox{if}\ u_{\lambda}(z)<x.
		  \end{array}\right.
	\end{eqnarray}
	
	We set $B(z,x)=\int^x_0\beta(z,s)ds$ and consider the $C^1$-functional $\sigma:W^{1,p}_0(\Omega)\rightarrow\RR$ defined by
	$$\sigma(u)=\frac{1}{p}||Du||^p_p-\int_{\Omega}B(z,u)dz\ \mbox{for all}\ u\in W^{1,p}_0(\Omega).$$
	
	From (\ref{eq26}) it is clear that $\sigma(\cdot)$ is coercive. Also, it is sequentially weakly lower semicontinuous. So, we can find $\bar{u}\in W^{1,p}_0(\Omega)$ such that
	\begin{eqnarray}\label{eq27}
		&&\sigma(\bar{u})=\inf\{\sigma(u):u\in W^{1,p}_0(\Omega)\},\nonumber\\
		&\Rightarrow&\sigma'(\bar{u})=0,\nonumber\\
		&\Rightarrow&\left\langle A(\bar{u}),h\right\rangle=\int_{\Omega}\beta(z,\bar{u})hdz\ \mbox{for all}\ h\in W^{1,p}_0(\Omega).
	\end{eqnarray}
	
	In (\ref{eq27}) we first choose $h=(\hat{y}_1-\bar{u})^+\in W^{1,p}_0(\Omega)$. Then
	\begin{eqnarray*}
		\left\langle A(\bar{u}),(\hat{y}_1-\bar{u})^+\right\rangle &=&\int_{\Omega}[\hat{\lambda}_1+\epsilon]\hat{y}_1^{p-1}(\hat{y}_1-\bar{u})^+dz\ (\mbox{see (\ref{eq26})})\\
		&\geq&\left\langle A(\hat{y}_1),(\hat{y}_1-\hat{u})^+\right\rangle\ (\mbox{see (\ref{eq25})}),\\
		\Rightarrow\hat{y}_1\leq\bar{u}.&
	\end{eqnarray*}
	
	Also, in (\ref{eq27}) we choose $h=(\bar{u}-u_{\lambda})^+\in W^{1,p}_0(\Omega)$. Then
	\begin{eqnarray*}
		\left\langle A(\bar{u}),(\bar{u}-u_{\lambda})^+\right\rangle&=&\int_{\Omega}[\hat{\lambda}_1+\epsilon]u_{\lambda}^{p-1}(\bar{u}-u_{\lambda})^+dz\ (\mbox{see (\ref{eq26})})\\
		&\leq&\left\langle A(u_{\lambda}),(\bar{u}-u_{\lambda})^+\right\rangle\ (\mbox{see (\ref{eq23})}),\\
		\Rightarrow\bar{u}\leq u_{\lambda}.&&
	\end{eqnarray*}
	
	So, we have proved that
	\begin{equation}\label{eq28}
		\bar{u}\in[\hat{y}_1,u_{\lambda}].
	\end{equation}
	
It follows	from (\ref{eq26}), (\ref{eq27}) and (\ref{eq28}) that
	\begin{eqnarray*}
		&&-\Delta_p\bar{u}(z)=[\hat{\lambda}_1+\epsilon]\bar{u}(z)^{p-1}\ \mbox{for almost all}\ z\in\Omega,\ \bar{u}|_{\partial\Omega}=0,\\
		&\Rightarrow&\bar{u}\in C^1_0(\overline{\Omega})\ \mbox{must be nodal, a contradiction (see (\ref{eq28}))}.
	\end{eqnarray*}
	
	Therefore we have $\lambda^*\leq\tilde{\lambda}<+\infty$.
\end{proof}

Next, we show that the critical parameter $\lambda^*>0$ is admissible.
\begin{prop}\label{prop12}
	If hypotheses $H(f)$ hold, then $\lambda^*\in\mathcal{L}$.
\end{prop}
\begin{proof}
	Let $\{\lambda_n\}_{n\geq1}\subseteq(0,\lambda^*)$ and assume that $\lambda_n\rightarrow(\lambda^*)^{-}$ as $n\rightarrow\infty$. We can find $u_n=u_{\lambda_n}\in S_{\lambda_n}\subseteq {\rm int}\,C_+$ for all $n\in\NN$. Then
	\begin{equation}\label{eq29}
		\langle A(u_n),h\rangle = \int_\Omega[\lambda_n u_n^{-\gamma}+f(z,u_n)]hdz\ \mbox{for all}\ h\in W^{1,p}_0(\lambda),\ \mbox{all}\ n\in\NN.
	\end{equation}
	
	Suppose that $||u_n||\rightarrow\infty$. We set $y_n=\frac{u_n}{||u_n||}\ n\in\NN$. Then $||y_n||=1, y_n\geq0$ for all $n\in\NN$. So, we may assume that
	\begin{equation}\label{eq30}
		y_n\xrightarrow{w}y\ \mbox{in}\ W^{1,p}_0(\Omega)\ \mbox{and}\ y_n\rightarrow y\ \mbox{in}\ L^p(\Omega)\ \mbox{as}\ n\rightarrow\infty.
	\end{equation}
	
	From (\ref{eq29}) we have
	\begin{equation}\label{eq31}
		\langle A(y_n),h\rangle = \int_\Omega\left[\frac{\lambda_n}{||u_n||^{p+\gamma-1}}y^{-\gamma}_n + \frac{N_f(u_n)}{||u_n||^{p-1}}\right]hdz\ \mbox{for all}\ h\in W^{1,p}_0(\Omega),\ n\in\NN.
	\end{equation}
	
	Hypotheses $H(f)(i),(ii)$ imply that
	$$
	|f(z,x)|\leq c_7[1+x^{p-1}]\ \mbox{for almost all}\ z\in\Omega,\ \mbox{all}\ x\geq0,\ \mbox{and some}\ c_7>0.
	$$
	
	This growth condition implies that
	\begin{equation}\label{eq32}
		\left\{\frac{N_f(u_n)}{||u_n||^{p-1}}\right\}_{n\geq1}\subseteq L^{p'}(\Omega)\ \mbox{is bounded}.
	\end{equation}
	
	Then (\ref{eq32}) and hypothesis $H(f)(ii)$ imply that at least for a subsequence, we have
	\begin{eqnarray}\label{eq33}
		\frac{N_f(u_n)}{||u_n||^{p-1}}\xrightarrow{w}\eta_0(z)y^{p-1}\ \mbox{in}\ L^{p'}(\Omega)\ \mbox{as}\ n\rightarrow\infty,\\
		\mbox{with}\ \eta\leq\eta_0(z)\leq\hat{\eta}\ \mbox{for almost all}\ z\in\Omega \nonumber \\
		\mbox{(see Aizicovici, Papageorgiou \& Staicu \cite{1}, proof of Proposition 16)}. \nonumber
	\end{eqnarray}
	
	In (\ref{eq31}) we choose $h=y_n-y\in W^{1,p}_0(\Omega)$, pass to the limit as $n\rightarrow\infty$, and use (\ref{eq30}) and (\ref{eq32}). Then
	\begin{eqnarray}
		& \lim_{n\rightarrow\infty}\langle A(y_n),y_n-y\rangle=0, \nonumber \\
		\Rightarrow & y_n\rightarrow y\ \mbox{in}\ W^{1,p}_0(\Omega)\ \mbox{(see Proposition \ref{prop3}), hence}\ ||y||=1,\ y\geq0. \label{eq34}
	\end{eqnarray}
	
	Therefore, if in (\ref{eq31}) we pass to the limit as $n\rightarrow\infty$ and use (\ref{eq34}) and (\ref{eq33}), then
	\begin{eqnarray}
		& \langle A(y),h\rangle = \int_\Omega\eta_0(z)y^{p-1}hdz\ \mbox{for all}\ h\in W^{1,p}_0(\Omega), \nonumber \\
		\Rightarrow & -\Delta_py(z)=\eta_0(z)y(z)^{p-1}\ \mbox{for almost all}\ z\in\Omega, \ y|_{\partial\Omega}=0. \label{eq35}
	\end{eqnarray}
	
	Since $\eta\leq\eta_0(z)\leq\hat{\eta}$ for almost all $z\in\Omega$ (see (\ref{eq33})), using Proposition \ref{prop4}, we have
	$$
	\tilde{\lambda}_1(\eta_0)\leq\tilde{\lambda}_1(\eta)<\tilde{\lambda}_1(\hat{\lambda_1})=1.
	$$
	
	So, from (\ref{eq35}) and since $||y||=1$ (see (\ref{eq34})), it follows that $y$ must be nodal, a contradiction (see (\ref{eq34})). Therefore
	$$
	\{u_n\}_{n\geq1}\subseteq W^{1,p}_0(\Omega)\ \mbox{is bounded}.
	$$
	
	Hence, we may assume that
	\begin{equation}\label{eq36}
		u_n\xrightarrow{w}u^*\ \mbox{in}\ W^{1,p}_0(\Omega)\ \mbox{and}\ u_n\rightarrow u^*\ \mbox{in}\ L^p(\Omega)\ \mbox{as}\ n\rightarrow\infty.
	\end{equation}
	
	On account of Corollary \ref{cor9}, we may assume that $\{u_n\}_{n\geq1}$ is nondecreasing. Therefore $u^*\neq0$. Also, we have
	\begin{equation}\label{eq37}
		0\leq(u^*)^{-\gamma}\leq u_n^{-\gamma}\leq u_1^{-\gamma}\in L^{p'}(\Omega)\ \mbox{for all}\ n\in\NN.
	\end{equation}
	
	From (\ref{eq36}) and by passing to a subsequence if necessary, we can say that
	\begin{equation}\label{eq38}
		u_n(z)^{-\gamma}\rightarrow u^*(z)^{-\gamma}\ \mbox{for almost all}\ z\in\Omega.
	\end{equation}
	
	From (\ref{eq37}), (\ref{eq38}) and Problem 1.19 of Gasinski \& Papageorgiou \cite{4}, we have that
	\begin{equation}\label{eq39}
		u_n^{-\gamma}\xrightarrow{w}(u^*)^{-\gamma}\ \mbox{in}\ L^{p'}(\Omega)\ \mbox{as}\ n\rightarrow\infty.
	\end{equation}
	
	If in (\ref{eq29}) we choose $h=u_n-u^*\in W^{1,p}_0(\Omega)$, pass to the limit as $n\rightarrow\infty$ and use (\ref{eq39}) and the fact that $\{N_f(u_n)\}_{n\geq1}\subseteq L^{p'}(\Omega)$ is bounded, then
	\begin{eqnarray}
		& & \lim_{n\rightarrow\infty}\langle A(u_n),u_n-u^*\rangle=0, \nonumber \\
		& \Rightarrow & u_n\rightarrow u^*\ \mbox{in}\ W^{1,p}_0(\Omega)\ \mbox{(see Proposition \ref{prop3})}. \label{eq40}
	\end{eqnarray}
	
	Finally, in (\ref{eq29}) we pass to the limit as $n\rightarrow\infty$ and use (\ref{eq39}) and (\ref{eq40}). We obtain
	\begin{eqnarray*}
		& & \langle A(u^*),h\rangle = \int_\Omega[\lambda^*(u^*)^{-\gamma} + f(z,u^*)]hdz\ \mbox{for all}\ h\in W^{1,p}_0(\Omega), \\
		& \Rightarrow & u^*\in S_{\lambda^*}\subseteq {\rm int}\,C_+\ \mbox{and}\ \lambda^*\in\mathcal{L}.
	\end{eqnarray*}
This completes the proof.
\end{proof}
We have proved that
$$
\mathcal{L}=(0,\lambda^*].
$$	
\begin{prop}\label{prop13}
	If hypotheses $H(f)$ hold and $\lambda\in(0,\lambda^*)$, then problem  $(P_\lambda)$ admits at least two positive solutions
	$$
	u_\lambda,\hat{u}_\lambda\in {\rm int}\,C_+, \hat{u}_{\lambda}-u_\lambda\in C_+\backslash\{0\}.
	$$
\end{prop}
\begin{proof}
	Let $u^*\in S_{\lambda^*}\subseteq {\rm int}\,C_+$ (see Proposition 12). Invoking Proposition \ref{prop10}, we can find $u_\lambda\in S_\lambda\subseteq {\rm int}\,C_+$ such that
	\begin{equation}\label{eq41}
		u^*-u_\lambda\in {\rm int}\,C_+.
	\end{equation}
	
	We consider the Carath\'eodory function $\tau_\lambda(z,x)$ defined by
	\begin{equation}\label{eq42}
		\tau_\lambda(z,x)=\left\{
			\begin{array}{ll}
				\lambda u_\lambda(z)^{-\gamma} + f(z,u_\lambda(z)) & \mbox{if}\ x\leq u_\lambda(z) \\
				\lambda x^{-\gamma} + f(z,x) & \mbox{if}\ u_\lambda(z)<x.
			\end{array}
		\right.
	\end{equation}
	
	Recall that $u_\lambda^{-\gamma}\in L^{p'}(\Omega)$ (see the proof of Proposition \ref{prop5}). We set $T_\lambda(z,x)=\int^x_0\tau_\lambda(z,s)ds$ and consider the functional $\tilde\varphi_\lambda:W^{1,p}_0(\Omega)\rightarrow\RR$ defined by
 	$$
	\tilde{\varphi}_\lambda(u)=\frac{1}{p}||Du||^p_p - \int_\Omega T_\lambda(z,u)dz\ \mbox{for all}\ u\in W^{1,p}_0(\Omega).
	$$
	
	We know that $\tilde{\varphi}_\lambda\in C^1(W^{1,p}_0(\Omega))$. Let $K_{\tilde{\varphi}_\lambda}=\{u\in W^{1,p}_0(\Omega):\tilde{\varphi}_\lambda'(u)=0\}$ (the critical set of $\tilde{\varphi}_\lambda$). Also, for $u\in W^{1,p}_0(\Omega)$, we set $$[u)=\{v\in W^{1,p}(\Omega):u(z)\leq v(z)\ \mbox{for almost all}\ z\in\Omega\}.$$
	\begin{claim}\label{claim1}
		$K_{\tilde{\varphi_\lambda}}\subseteq[u_\lambda)\cap {\rm int}\,C_+$.
	\end{claim}

	Let $u\in K_{\tilde{\varphi}_\lambda}$. We have
	\begin{equation}\label{eq43}
		\langle A(u),h\rangle = \int_\Omega \tau_\lambda(z,u)hdz\ \mbox{for all}\ h\in W^{1,p}_0(\Omega).
	\end{equation}
	
	We choose $h=(u_\lambda-u)^+\in W^{1,p}_0(\Omega)$. Then
	\begin{eqnarray}
		\langle A(u),(u_\lambda-u)^+\rangle & = & \int_\lambda [\lambda u_\lambda^{-\gamma} + f(z,u_\lambda)](u_\lambda-u)^+dz\ \mbox{(see (\ref{eq42}))} \nonumber \\
		& = & \langle A(u_\lambda), (u_\lambda-u)^+\rangle\ \mbox{(since $u_\lambda\in S_\lambda$)}, \nonumber \\
		& \Rightarrow & u_\lambda\leq u. \label{eq44}
	\end{eqnarray}
	
	From (\ref{eq42}), (\ref{eq43}) and (\ref{eq44}), we obtain
	\begin{eqnarray*}
		& & \langle A(u),h\rangle = \int_\Omega [\lambda u^{-\gamma} + f(z,u)]hdz\ \mbox{for all}\ h\in W^{1,p}_0(\Omega), \\
		& \Rightarrow & u\in S_\lambda\subseteq {\rm int}\,C_+\ \mbox{and}\ u_\lambda\leq u, \\
		& \Rightarrow & u\in[u_\lambda)\cap {\rm int}\,C_+.
	\end{eqnarray*}
	This proves Claim 1.
	
	Note that $u_\lambda\in K_{\tilde{\varphi}_\lambda}$. We may assume that
	\begin{equation}\label{eq45}
		K_{\tilde{\varphi}_\lambda}\cap[u_\lambda,u^*]=\{u_\lambda\},
	\end{equation}
	or otherwise we already have a second positive smooth solution for problem \eqref{eqp} (see (\ref{eq42})) and so we are done.
	
	We introduce the following Carath\'eodory function
	\begin{equation}\label{eq46}
		\hat{\tau}_\lambda(z,x)=\left\{
		\begin{array}{ll}
			\tau_\lambda(z,x) & \mbox{if}\ x\leq u^*(z) \\
			\tau_\lambda(z,u^*(z)) & \mbox{if}\ u^*(z)<x.
		\end{array}
		\right.
	\end{equation}
	
	We set $\hat{T_\lambda}(z,x)=\int^x_0\hat{\tau}_\lambda(z,s)ds$ and consider the $C^1$-functional $\hat{\varphi}_\lambda:W^{1,p}(\Omega)\rightarrow\RR$ defined by
	$$
	\hat{\varphi}_\lambda(u)=\frac{1}{p}||Du||^p_p - \int_\Omega\hat{T_\lambda}(z,u)dz\ \mbox{for all}\ u\in W^{1,p}_0(\Omega).
	$$
	
	This functional is coercive (see (\ref{eq46})) and sequentially weakly lower semicontinuous. Hence we can find $\tilde{u}_\lambda\in W^{1,p}_0(\Omega)$ such that
	\begin{eqnarray}
		& & \hat{\varphi}_\lambda(\tilde{u}_\lambda) = \inf\{\hat{\varphi}_\lambda(u):u\in W^{1,p}_0(\Omega)\}, \nonumber \\
		& \Rightarrow & \hat{\varphi}'_{\lambda}(\tilde{u}_\lambda)=0, \nonumber \\
		& \Rightarrow & \langle A(\tilde{u}_\lambda),h\rangle = \int_\Omega\hat{\tau}_\lambda(z,\tilde{u}_\lambda)hdz\ \mbox{for all}\ h\in W^{1,p}_0(\Omega). \label{eq47}
	\end{eqnarray}
	
	In (\ref{eq47}) we choose $h=(u_\lambda-\tilde{u}_\lambda)^+\in W^{1,p}_0(\Omega)$ and $h=(\tilde{u}_\lambda-u^*)^+\in W^{1,p}_0(\Omega)$ and obtain that
	\begin{equation}\label{eq48}
		\tilde{u}_\lambda\in[u_\lambda,u^*].
	\end{equation}
	
	From (\ref{eq46}), (\ref{eq47}), (\ref{eq48}) we infer that
	\begin{eqnarray*}
		& & \tilde{u}_\lambda\in K_{\tilde{\varphi}_\lambda}\cap[u_\lambda,u^*], \\
		& \Rightarrow & \tilde{u}_\lambda = u_\lambda\ \mbox{(see (\ref{eq45}))}.
	\end{eqnarray*}
	
	From (\ref{eq42}) and (\ref{eq46}) it is clear that
	$$
	\tilde{\varphi}_\lambda|_{[0,u^*]}=\hat{\varphi}_\lambda|_{[0,u^*]}.
	$$
	
	Also, $u_\lambda$ is a minimizer of $\hat{\varphi}_\lambda$. Since $u^*-u_\lambda\in {\rm int}\,C_+$ (see (\ref{eq41})), it follows that
	\begin{eqnarray}
		& & u_\lambda\ \mbox{is a local}\ C^1_0(\overline{\Omega})-\mbox{minimizer of}\ \tilde{\varphi}_\lambda, \nonumber \\
		& \Rightarrow & u_\lambda\ \mbox{is a local}\ W^{1,p}_0(\Omega)-\mbox{minimizer of}\ \tilde{\varphi}_\lambda. \label{eq49} \\
		&& \mbox{(see Motreanu, Motreanu \& Papageorgiou \cite[Theorem 12.18, p. 409]{15})}. \nonumber
	\end{eqnarray}
	
	We assume that $K_{\tilde{\varphi}_\lambda}$ is finite or otherwise on account of Claim \ref{claim1}, we already have an infinity of positive smooth solutions for problem \eqref{eqp} bigger than $u_\lambda$ and so we are done. Because of (\ref{eq49}), we can find $\rho\in(0,1)$ small such that
	\begin{eqnarray}
		\tilde{\varphi}_\lambda(u_\lambda)<\inf\{\tilde{\varphi}_\lambda(u):||u-u_\lambda||=\rho\}=\tilde{m}_\lambda \label{eq50} \\
		\mbox{(see Aizicovici, Papageorgiou \& Staicu \cite{1}, proof of Proposition 29)}. \nonumber
	\end{eqnarray}
	Hypothesis $H(f)(ii)$ implies that
	\begin{equation}\label{eq51}
		\tilde{\varphi}_\lambda(t\hat{u}_1)\rightarrow-\infty\ \mbox{as}\ t\rightarrow+\infty.
	\end{equation}
	
	\begin{claim}\label{claim2}
		$\tilde{\varphi}_\lambda$ satisfies the $C$-condition.
	\end{claim}
	
	Let $\{u_n\}_{n\geq1}\subseteq W^{1,p}_0(\Omega)$ such that $\{\tilde{\varphi}_\lambda(u_n)\}_{n\geq1}\subseteq\RR$ is bounded and
	$$
	(1+||u_n||)\tilde{\varphi}_\lambda'(u_n)\rightarrow0\ \mbox{in}\ W^{-1,p'}(\Omega)=W^{1,p}_0(\Omega)^*\ \mbox{as}\ n\rightarrow\infty.
	$$
	
	We have
	\begin{equation}\label{eq52}
		|\langle A(u_n),h\rangle - \int_\Omega\tau_\lambda(z,u_n)hdz| \leq \frac{\varepsilon_n||h||}{1+||u_n||}\ \mbox{for all}\ h\in W^{1,p}_0(\Omega),\ \mbox{with}\ \varepsilon_n\rightarrow0^+.
	\end{equation}
	
	We choose $h=-u^-_n\in W^{1,p}_0(\Omega)$ in (\ref{eq52}) and also use (\ref{eq42}). Then
	\begin{eqnarray}
		& & ||Du^-_n||^p_p \leq c_8||u^-_n||\ \mbox{for some}\ c_8>0, \mbox{and all}\ n\in\NN, \nonumber \\
		& \Rightarrow & \{u^-_n\}_{n\geq1}\subseteq W^{1,p}_0(\Omega)\ \mbox{is bounded}. \label{eq53}
	\end{eqnarray}
	
	Suppose that $||u^+_n||\rightarrow\infty$ and let $y_n=\frac{u^+_n}{||u^+_n||}\ n\in\NN$. Then $||y_n||=1,y_n\geq0$ for all $n\in\NN$. So, we may assume that
	\begin{equation}\label{eq54}
		y_n\xrightarrow{w}y\ \mbox{in}\ W^{1,p}_0(\Omega)\ \mbox{and}\ y_n\rightarrow y\ \mbox{in}\ L^p(\Omega),\ y\geq0.
	\end{equation}
	
	From (\ref{eq52}) and (\ref{eq53}), we have
	\begin{equation}\label{eq55}
		|\langle A(y_n),h\rangle - \int_\Omega\frac{N_{\tau_\lambda}(u_n^+)}{||u_n^+||^{p-1}}hdz| \leq \varepsilon_n'||h||\ \mbox{for all}\ h\in W^{1,p}_0(\Omega),\ \mbox{with}\ \varepsilon_n'\rightarrow0.
	\end{equation}
	From (\ref{eq42}) and hypothesis $H(f)(ii)$, we have
	\begin{eqnarray}\label{eq56}
		\frac{N_{\tau_\lambda}(u_n^+)}{||u_n^+||^{p-1}}\xrightarrow{w} \eta_0(z)y^{p-1}\ \mbox{in}\ L^{p'}(\Omega)\ \mbox{as}\ n\rightarrow\infty\ \\
		\mbox{with}\ \eta\leq\eta_0(z)\leq\hat{\eta}\ \mbox{for almost all}\ z\in\Omega.\ \mbox{(see (\ref{eq33}))}.\nonumber
	\end{eqnarray}
	
	In (\ref{eq55}) we choose $h=y_n-y\in W^{1,p}_0(\Omega)$ and pass to the limit as $n\rightarrow\infty$. Then
	\begin{eqnarray}
		& & \lim_{n\rightarrow\infty}\langle A(y_n),y_n-y\rangle=0, \nonumber \\
		& \Rightarrow & y_n\rightarrow y\ \mbox{in}\ W^{1,p}_0(\Omega)\ \mbox{(see Proposition \ref{prop3}), hence}\ ||y||=1, y\geq0. \label{eq57}
	\end{eqnarray}
	
	Then passing to the limit as $n\rightarrow\infty$ in (\ref{eq55}) and using (\ref{eq56}) and (\ref{eq57}), we obtain
	\begin{eqnarray}
		& & \langle A(y),h\rangle = \int_\Omega\eta_0(z)y^{p-1}hdz\ \mbox{for all}\ h\in W^{1,p}_0(\Omega), \nonumber \\
		& \Rightarrow & -\Delta_py(z)=\eta_0(z)y(z)^{p-1}\ \mbox{for almost all}\ z\in\Omega, y|_{\partial\Omega}=0. \label{eq58}
	\end{eqnarray}
	
	As before, using Proposition 4, we have
	\begin{eqnarray*}
		& & \tilde{\lambda}_1(\eta_0)\leq\tilde{\lambda}_1(\eta)<\tilde{\lambda}_1(\hat{\lambda}_1)=1, \\
		& \Rightarrow & y\ \mbox{must be nodal (see (\ref{eq58}), (\ref{eq57})), a contradiction (see (\ref{eq57}))}.
	\end{eqnarray*}
	
	This proves that $\{u^+_n\}_{n\geq1}\subseteq W^{1,p}_0(\Omega)$ is bounded. Hence
	$$
	\{u_n\}_{n\geq1}\subseteq W^{1,p}_0(\Omega)\ \mbox{is bounded (see (\ref{eq53}))}.
	$$
	
	So, we may assume that
	\begin{equation}\label{eq59}
		u_n\xrightarrow{w}u\ \mbox{in}\ W^{1,p}_0(\Omega)\ \mbox{and}\ u_n\rightarrow u\ \mbox{in}\ L^p(\Omega)\ \mbox{as}\ n\rightarrow\infty.
	\end{equation}
	
	In (\ref{eq52}) we choose $h=u_n-u\in W^{1,p}_0(\Omega)$, pass to the limit as $n\rightarrow\infty$ and use (\ref{eq59}). Then
	\begin{eqnarray*}
		& & \lim_{n\rightarrow\infty}\langle A(u_n),u_n-u\rangle = 0, \\
		& \Rightarrow & u_n\rightarrow u\ \mbox{in}\ W^{1,p}_0(\Omega)\ \mbox{(see Proposition 3)}.
	\end{eqnarray*}
	
	This proves Claim \ref{claim2}.
	
	On account of (\ref{eq50}), (\ref{eq51}) and Claim \ref{claim2} we can apply Theorem \ref{th1} (the mountain pass theorem) and find $\hat{u}_\lambda\in W^{1,p}_0(\Omega)$ such that
	\begin{eqnarray*}
		\hat{u}_\lambda\in K_{\tilde{\varphi}_\lambda}\subseteq[u_\lambda)\cap {\rm int}\,C_+\ \mbox{(see Claim \ref{claim1})}, \\
		\tilde{m_\lambda}\leq\tilde{\varphi}_\lambda(\hat{u}_\lambda)\ \mbox{(see (\ref{eq50})), hence}\ \hat{u}_\lambda\neq u_\lambda.
	\end{eqnarray*}
	
	Therefore $\hat{u}_\lambda\in {\rm int}\,C_+$ is the second positive solution of \eqref{eqp} and
	$$
	\hat{u}_\lambda - u_\lambda\in C_+\backslash\{0\}.
	$$
The proof is now complete.
\end{proof}	

\textcolor{black}{Therefore we have also proved Theorem A, which is the main result of this paper.}

\begin{remark}
	An interesting open problem is whether there is such a bifurcation-type theorem for resonant problems, that is,
	$$
	\hat{\lambda}_1\leq\liminf_{x\rightarrow +\infty}\frac{f(z,x)}{x^{p-1}} \leq\limsup_{x\rightarrow +\infty}\frac{f(z,x)}{x^{p-1}}\leq\hat{\eta}\ \mbox{uniformly for almost all}\ z\in\Omega
	$$
	or even for the nonuniformly nonresonant problems, that is,
	$$\eta(z)\leq\liminf\limits_{x\rightarrow+\infty}\frac{f(z,x)}{x^{p-1}}\leq\limsup\limits_{x\rightarrow +\infty}\frac{f(z,x)}{x^{p-1}}\leq\hat{\eta}\ \mbox{uniformly for almost all}\ z\in\Omega$$
	with $\eta\in L^\infty(\Omega)$ such that
	$$
	\hat{\lambda}_1\leq\eta(z)\ \mbox{for almost all}\ z\in\Omega,\ \eta\not\equiv\hat{\lambda}_1.
	$$
\end{remark}

In both cases it seems to be difficult to show that $\lambda^*<\infty$. Additional conditions on $f(z,\cdot)$ might be needed.

\medskip
{\bf Acknowledgements.} \textcolor{black}{The authors wish to thank the referee for his/her remarks and suggestions.}
This research was supported by the Slovenian Research Agency grants
P1-0292, J1-8131, J1-7025, N1-0064, and N1-0083. V.D.~R\u adulescu acknowledges the support through a grant of the Romanian Ministry of Research and Innovation, CNCS-UEFISCDI, project number PN-III-P4-ID-PCE-2016-0130,
within PNCDI III.

\end{document}